\documentclass[12pt]{article}

\usepackage{amsmath}
\usepackage{amsthm}
\usepackage{amsfonts}
\usepackage{mathrsfs}
\usepackage{stmaryrd}
\usepackage{setspace}
\usepackage{fullpage}
\usepackage{amssymb}
\usepackage{breqn}
\usepackage{enumitem}
\usepackage{bbold}
\usepackage{authblk}
\usepackage{comment}
\usepackage{hyperref}
\usepackage{pgf,tikz}
\usepackage{graphicx}
\usepackage{subcaption}

\newtheorem{thm}{Theorem}[section]

\newtheorem{lemma}[thm]{Lemma}

\newtheorem{prop}[thm]{Proposition}

\newtheorem{defin}[thm]{Definition}

\newtheorem{clm}[thm]{Claim}

\newcommand\ex{\ensuremath{\mathrm{ex}}}

\newcommand\cF{{\mathcal F}}
\newcommand\cG{{\mathcal G}}
\newcommand\cH{{\mathcal H}}

\newcommand\cN{{\mathcal N}}

\newtheorem*{thm*}{Theorem}
\newtheorem*{prop*}{Proposition}
\newcommand{\ignore}[1]{}

\title{The Tur\'{a}n number of Berge paths}
\author{
Xin Cheng\thanks{\small School of Mathematics and Statistics, Northwestern Polytechnical University and Xi'an-Budapest Joint Research Center for Combinatorics, Xi'an 710129, Shaanxi, P.R. China. Email:
\small \texttt{xincheng@mail.nwpu.edu.cn}.}\,, \hspace{0.2em}
D\'{a}niel Gerbner\thanks{\small Alfr\'ed R\'enyi Institute of Mathematics. Email:
\small \texttt{gerbner.daniel@renyi.hu}.}\,, \hspace{0.2em} 
Hilal Hama Karim$^{\dagger,}$\thanks{Department of Computer Science and Information Theory, Faculty of Electrical Engineering and Informatics, Budapest University of Technology and Economics, Műegyetem rkp. 3., H-1111 Budapest, Hungary. E-mail: \texttt{hilal.hamakarim@edu.bme.hu}.}\,, \hspace{0.2em}

Shujing Miao\thanks{\small School of Mathematics and Statistics, and Hubei Key Lab--Math. Sci., Central China Normal University, Wuhan 430079, China. Email:
\small \texttt{sjmiao2020@sina.com}.}\,, \hspace{0.2em}
Junpeng Zhou\thanks{\small \textit{Corresponding author.} Department of Mathematics, Shanghai University, Shanghai 200444, P.R. China. Email:
\small \texttt{junpengzhou@shu.edu.cn}.} \thanks{\small Newtouch Center for Mathematics of Shanghai University, Shanghai 200444, P.R. China.}}

\date{}

\begin{document}

\maketitle

\begin{abstract} 
A Berge path of length $k$ in an $r$-uniform hypergraph is a collection of $k$ hyperedges $h_1,\dots,h_k$ and $k+1$ vertices $v_1,\dots,v_{k+1}$ such that $v_i, v_{i+1}\in h_i$ for each $1\le i\le k$. Gy\H{o}ri, Katona and Lemons [\textit{European J. Combin. 58 (2016) 238--246}] generalized the Erd\H{o}s-Gallai theorem to Berge paths and established bounds for the Tur\'{a}n number of Berge paths. However, these bounds are sharp only when some divisibility conditions hold. Gy\H ori, Lemons, Salia and Zamora [\textit{J. Combin. Theory Ser. B 148 (2021) 239--250}] determined the exact value of the Tur\'{a}n number of Berge paths in the case $k\le r$. In this paper, we settle the final open case $k>r$, thereby completing the determination of the Tur\'{a}n number of Berge paths.  
\end{abstract}

{\noindent{\bf Keywords}: Tur\'{a}n number, hypergraph, Berge path}

{\noindent{\bf AMS subject classifications:} 05C35, 05C65}

\section{Introduction}
A \textit{hypergraph} $\cH=(V(\cH),E(\cH))$ consists of a vertex set $V(\cH)$ and a hyperedge set $E(\cH)$, where each hyperedge in $E(\cH)$ is a nonempty subset of $V(\cH)$. If $|e|=r$ for every $e\in E(\cH)$, then $\cH$ is called an \textit{$r$-uniform hypergraph} ($r$-graph for short). For simplicity, let $e(\cH):=|E(\cH)|$. The \textit{degree} $d_{\cH}(v)$ of a vertex $v$ is the number of hyperedges containing $v$ in $\cH$. Let $\cH_1\cup \cH_2$ denote the disjoint union of hypergraphs $\cH_1$ and $\cH_2$, and $k\cH$ denote the disjoint union of $k$ hypergraphs $\cH$. 

Let $\mathcal{F}$ be a family of $r$-graphs. An $r$-graph $\cH$ is called \textit{$\mathcal{F}$-free} if $\cH$ does not contain any member in $\mathcal{F}$ as a subhypergraph. The \textit{Tur\'{a}n number} ${\rm{ex}}_r(n,\mathcal{F})$ of $\mathcal{F}$ is the maximum number of hyperedges in an $\mathcal{F}$-free $r$-graph on $n$ vertices. If $\mathcal{F}=\{\cG\}$, then we write ${\rm{ex}}_r(n,\cG)$ instead of ${\rm{ex}}_r(n,\{\cG\})$. When $r=2$, we write $\mathrm{ex}(n,\mathcal{F})$ instead of $\mathrm{ex}_2(n,\mathcal{F})$. The study of this function is a central theme in extremal combinatorics. Tur\'{a}n's theorem \cite{Tu} determines its exact value when $\cG$ is a complete graph $K_k$, and the Erd\H{o}s-Stone-Simonovits theorem \cite{ESi,ESt} settles asymptotically the case when a graph $\cG$ has chromatic number at least 3. The Erd\H{o}s-Gallai theorem states that a graph of order $n$ containing no $P_k$ as a subgraph contains at most $\frac{(k-1)n}{2}$ edges, where $P_k$ denote a path with $k$ edges. In 2016, Gy\H{o}ri, Katona and Lemons \cite{B6} defined the concept of Berge paths and generalized the Erd\H{o}s-Gallai theorem to Berge paths. 

\begin{defin}[Gy\H{o}ri, Katona and Lemons \cite{B6}]
A Berge path of length $k$ in an $r$-uniform hypergraph is a collection of $k$ hyperedges $h_1,\dots,h_k$ and $k+1$ vertices $v_1,\dots,v_{k+1}$ such that for each $1\le i\le k$ we have $v_i, v_{i+1}\in h_i$. 
\end{defin}

We refer to $\{v_1,\dots,v_{k+1}\}$ as the set of \textit{defining vertices} and to $\{h_1,\dots h_k\}$ as the set of \textit{defining hyperedges}. Note that for a fixed path $P_k$ there are many hypergraphs that are a Berge-$P_k$. For convenience, we refer to this collection of hypergraphs as ``Berge-$P_k$''. 

The Turán number of Berge paths of length $k$ in $r$-graphs was studied by Gy\H{o}ri, Katona and Lemons \cite{B6}, who proved sharp bounds for almost every pair $k,r$. The missing case when $k=r+1>2$ was settled by Davoodi, Gy\H{o}ri, Methuku and Tompkins \cite{10-10A1}.

\begin{thm}[Gy\H{o}ri, Katona, Lemons \cite{B6}, Davoodi, Gy\H{o}ri, Methuku, Tompkins \cite{10-10A1}]\label{gykl} 
\item[\textbf{(i)}] If $k\ge r+1>3$, then ${\rm{ex}}_r(n,{\rm Berge}{\text -}P_k)\leq \frac{n}{k}\binom{k}{r}$. Furthermore, this bound is sharp whenever $k$ divides $n$.
\item[\textbf{(ii)}] If $r\geq k>2$, then ${\rm{ex}}_r(n,{\rm Berge}{\text -}P_k)\leq \frac{n(k-1)}{r+1}$. Furthermore, this bound is sharp whenever $r+1$ divides $n$.
\end{thm}

Both upper bounds in Theorem \ref{gykl} are sharp as shown by the following examples. In the case when $k>r$, suppose that $k$ divides $n$ and partition the $n$ vertices into sets of size $k$. In each $k$-set, take all possible subsets of size $r$ as hyperedges of the hypergraph. The resulting $r$-graph has exactly $\frac{n}{k}\binom{k}{r}$ hyperedges and clearly contains no copy of any Berge-$P_k$. In the case when $k\le r$ and $r+1$ divides $n$, we partition the $n$ vertices into sets of size $r+1$, and from each $(r+1)$-set, select exactly $k-1$ of its subsets of size $r$ as hyperedges of the hypergraph. The resulting $r$-graph has exactly $\frac{k-1}{r+1}n$ hyperedges. Since each component contains exactly $k-1$ hyperedges, it is clear that there are no Berge-$P_k$. 

Observe however that these bounds are sharp only in the case the above divisibility conditions hold. Gy\H ori, Lemons, Salia and Zamora \cite{GLSZ} showed that ${\rm{ex}}_r(n,{\rm Berge}{\text -}P_k)=\left\lfloor \frac{n}{r+1}\right\rfloor(k-1)+\mathbf{1}_{r+1\,|\,n+1}$ if $3\le k\le r$, where $\mathbf{1}_{r+1\,|\,n+1}=1$ if $r+1\,|\,n+1$, and $\mathbf{1}_{r+1\,|\,n+1}=0$ otherwise. Note that if $k=2$, then ${\rm{ex}}_r(n,{\rm Berge}{\text -}P_2)=\left\lfloor \frac{n}{r}\right\rfloor$. 

Here, we complete the study of the Tur\'{a}n number of Berge paths by the following theorem. 

\begin{thm}\label{main}
    Let $k\ge r+1$ and $n=pk+q$ with $q<k$. Then $\ex_r(n,\textup{Berge-}P_k)=p\binom{k}{r}+\binom{q}{r}$.
\end{thm}

This also strengthens a result of Chakraborti and Chen \cite{chch}. 
Let $\mathcal{N}(G,H)$ denote the number of copies of $G$ in the graph $H$.

\begin{thm}[Chakraborti and Chen \cite{chch}]\label{cc}
For any positive integers $n$ and $3\leq r \leq k$, if $G$ is a $P_k$-free graph on $n$ vertices, then $\cN(K_r,G)\leq \cN(K_r,pK_{k}\cup K_q)$, where $n=pk+q$, $0\leq q\leq k-1$. 
\end{thm}

Given graphs $G$ and $F$, the \textit{generalized Tur\'{a}n number} $\ex(n,G,F)$ is the largest number of copies of $G$ in $F$-free $n$-vertex graphs. For a very recent survey on generalized Tur\'{a}n problems, one may refer to the work of Gerbner and Palmer \cite{GePa}. 

Berge \cite{A2} defined the Berge cycle, and a \textit{Berge cycle} of length $k$ in an $r$-uniform hypergraph is an alternating sequence of distinct vertices and hyperedges of the form $v_1, h_1, v_2, h_2,$ $\ldots, v_k, h_k, v_1$ where $v_i, v_{i+1} \in h_i$ for each $i \in \{1, 2, \ldots, k - 1\}$ and $v_k, v_1 \in h_k$. Similarly, we refer to $\{v_1,\dots,v_{k}\}$ as the set of \textit{defining vertices} and to $\{h_1,\dots h_k\}$ as the set of \textit{defining hyperedges}.
Gerbner and Palmer \cite{B3} generalized the established concepts of Berge cycle and Berge path to general graphs. Let $F$ be a graph. An $r$-graph $\cH$ is a \textit{Berge-$F$} if there is a bijection $\phi: E(F)\rightarrow E(\cH)$ such that $e\subseteq \phi(e)$ for each $e\in E(F)$. For a fixed graph $F$ there are many hypergraphs that are a Berge-$F$. Similarly, we refer to this collection of hypergraphs as ``Berge-$F$''. 

A simple connection between generalized Tur\'an problems and Berge hypergraphs is $\ex(n,K_r,F)\le \ex_r(n,\textup{Berge-}F)$. Indeed, if we take an $n$-vertex $F$-free graph with $\ex(n,K_r,F)$ copies of $K_r$ and add a hyperedge on each copy of $K_r$, then the resulting $r$-graph is clearly Berge-$F$-free. Therefore, Theorem \ref{main} is indeed a strengthening of Theorem \ref{cc}. Note that Chakraborti and Chen \cite{chch} also determined the cases where equality is achieved.

The rest of this paper is organized as follows. Similarly to Theorem \ref{gykl}, the case $k=r+1$ needs a different proof, which we give in Section \ref{smallk}. In Section \ref{bigk}, we prove Theorem \ref{main} in the case $k>r+1$.

\section{Proof of Theorem \ref{main} for $k=r+1$}\label{smallk}

\begin{lemma}\label{lemnew1}
    In a Berge-$P_{r+1}$-free, Berge-$C_{r+1}$-free $r$-graph $\cH$, the endpoints of a Berge-$P_r$ are contained only in hyperedges that are defining hyperedges of the Berge-$P_r$. 
    
\end{lemma}


\begin{proof}[\bf Proof] 
Consider the Berge path $P$ of length $r$ in $\cH$. Observe that any non-defining hyperedge that contains an endpoint of $P$ and a vertex outside the set of defining vertices of $P$ would create a Berge path of length $r+1$, and any non-defining hyperedge that contains both endpoints of $P$ would create a Berge cycle of length $r+1$.  
The only possible non-defining hyperedge $h$ is the one that contains all the defining vertices of $P$ except for an endpoint, say $v_1$. But then for the edge on the other end, say $v_rv_{r+1}$, we can change the defining hyperedge to $h$, and then we have a non-defining hyperedge containing an endpoint $v_{r+1}$. Based on the above analysis, this leads to a multiple hyperedge, and thus a contradiction. 
\end{proof}

\begin{lemma}\label{lemnew2}
    If a Berge-$P_{r+1}$-free, Berge-$C_{r+1}$-free $r$-graph $\cH$ contains a Berge-$C_r$, then it contains $r+1$ vertices that are contained in only $r$ hyperedges.
\end{lemma}
\begin{proof}[\bf Proof]
If there is a Berge cycle $C$ of length $r$, then all but at most one of the defining hyperedges of $C$ contain a vertex outside the set of defining vertices of $C$. For the defining hyperedge of the edge $v_iv_{i+1}$, let $x_i$ be this outside vertex. Then we can go from $v_{i}$ to $x_i$ instead of $v_{i+1}$, thus we found a Berge path with the same defining hyperedges and endpoints $v_{i}$, $x_i$. For all such defining vertices $v_{i}$ of the cycle and all such vertices $x_i$, we find such a Berge path of length $r$ that has $x_i$ as an endpoint. Thus, these at least $r+1$ vertices are contained only in the $r$ defining hyperedges of $C$ by Lemma \ref{lemnew1}. 
\end{proof}

For a set $S$ of vertices, let $N_\cH(S)$ denote the set of hyperedges that contain at least one vertex of $S$.


\begin{lemma}\label{lemnew3}
    Let $\cH$ be an $r$-graph with no Berge cycles of length at least $r$ such that the longest Berge path has length $r$. Then at least one of the following holds.

    \textbf{(i)} There is a set $S$ of size $r-1$ with $|N_\cH(S)|\le 1$.

    \textbf{(ii)} There is a set $S$ of size $r+1$ with $|N_\cH(S)|\le r+1$.
\end{lemma}


This is a variant of Lemma 1 in \cite{GLSZ}. There the assumption is that there is no Berge cycle of length at least $k$ (where $k \leq r$, for them) and it also works for multi-hypergraphs. Here we 
consider the case $k=r+1$ and simple hypergraphs, and also assume the condition on Berge paths. There are three possible consequences in \cite{GLSZ}. With our restrictions, \textbf{(i)} in \cite{GLSZ} is identical to the corresponding statement here, while \textbf{(ii)} in \cite{GLSZ} states that there is a set $S$ of size $r$ with $|N_\cH(S)|\le r-1$. There is a third possibility in \cite{GLSZ} that we do not describe here. We remark that simply applying Lemma 1 from \cite{GLSZ} is almost enough for us. The statements \textbf{(i)} and \textbf{(iii)} there could be applied in our proof. However, finding a set $S$ of size $r$ with $|N_\cH(S)|\le r-1$ is not suitable for our purposes, but we can use a set $S$ of size $r+1$ with $|N_\cH(S)|\le r+1$.

We follow the line of thought of the proof in \cite{GLSZ} below.

\begin{proof}[\bf Proof] 
    Let us choose a Berge path $v_1,e_1,v_2,\dots,e_r,v_{r+1}$ of length $r$. 
    Let $\cF=\{e_1,\dots,e_{r-1}\}$, $U=\{v_2,\dots,v_r\}$ and $U'=U\cup \{v_{r+1}\}$. Observe that $e_1$ does not contain $v_{r+1}$, since then $v_{r+1},e_1,v_2,\dots,e_r,v_{r+1}$ would form a Berge-$C_r$. Analogously, $v_1$ is not in $e_{r}$, hence by Lemma \ref{lemnew1} $v_1$ is contained only by hyperedges in $\cF$.

    \begin{clm}\label{claimnew}
        \textbf{(i)} $N_\cH(e_1\setminus U)\subseteq \cF$.

        \textbf{(ii)} If for some $2\le i\le r$ we have $v_i\in e_1$, then $N_\cH(e_{i-1}\setminus U)\subseteq \cF$.

        \textbf{(iii)} If there are two vertices $v_i,v_j\in e_1$ with $i>j$ such that $(e_{i-1}\cap e_{j-1})\setminus U'\neq \emptyset$, then
$N_\cH(v_{i-1})\subseteq \cF$ and $N_\cH(v_{j})\subseteq \cF$.

\textbf{(iv)} If for some $i$ we have 
$v_1\in e_{i-1}$, then $N_\cH(v_{i-1})\subseteq \cF$. 
    \end{clm}

    Note that \textbf{(i)}, \textbf{(ii)} and \textbf{(iii)} are equivalent to statements proved in \cite{GLSZ}, but for the sake of completeness, we prove them.

    \begin{proof}[\bf Proof of Claim]
        Any vertex of $e_1\setminus U$ can replace $v_1$ to be the endvertex of a Berge path of length $r$ with defining hyperedges $\cF\cup\{e_r\}$, thus we can apply the argument above the claim to show that they are contained only in hyperedges of $\cF$. This proves \textbf{(i)}.  
        
        For $w\in e_{i-1}\setminus U$, consider $w,e_{i-1},v_{i-1},\dots, v_2,e_1,v_i,e_{i}, \dots, e_r,v_{r+1}$. If $w=v_{r+1}$, this is a Berge-$C_r$, a contradiction. Otherwise, this is a Berge-$P_r$ with $w$ as an endpoint and $\cF\cup \{e_r\}$ as defining hyperedges. If $e_r$ contains $w$, we clearly obtain a Berge-$C_r$, completing the proof of \textbf{(ii)}.

        Fix $u\in (e_{i-1}\cap e_{j-1})\setminus U'$, and consider the Berge paths $v_{i-1},e_{i-2},v_{i-2}\dots, e_{j+1},v_j,e_1,v_2,e_2, \\ \dots, e_{j-1},u,e_{i-1},v_{i},e_{i},\dots, e_r,v_{r+1}$ and $v_j,e_j,v_{j+1},\dots,v_{i-1},e_{i-1},u,e_{j-1},v_{j-1},\dots, v_2,e_1,v_i,e_i, \\ \dots, e_r,v_{r+1}$. Both $v_{i-1}$ and $v_j$ are endpoints of Berge-$P_r$ with defining hyperedges in $\cF \cup \{e_r\}$, thus Lemma \ref{lemnew1} completes the proof of \textbf{(iii)}.

        Finally, to prove \textbf{(iv)}, the Berge path $v_{i-1},e_{i-2},\dots, v_2,e_1,v_1,e_{i-1},v_i,e_i,\dots, e_r,v_{r+1}$ 
        with Lemma \ref{lemnew1} completes the proof.
    \end{proof}

    Let us return to the proof of the lemma.

    \begin{clm}
        If $e_1\cap U=\{v_2\}$, then either \textbf{(i)} or \textbf{(ii)} of the Lemma holds.
    \end{clm}

    \begin{proof}[\bf Proof of Claim]
    By Claim \ref{claimnew} \textbf{(i)}, $N_\cH(e_1\setminus U)\subseteq \cF$. If there is no $w\in e_1\setminus U$ that is contained in some $e_i$ with $i>1$, then \textbf{(i)} of the Lemma holds. Otherwise, the Berge-path $v_{i},e_{i-1},\dots, v_2,e_1,w,e_i, v_{i+1},e_{i+1},\dots,e_r,v_{r+1}$ shows that $N_\cH(v_i)\subseteq \cF\cup \{e_r\}$. Finally, $N_\cH(v_{r+1})\subseteq \cF\cup \{e_r\}$ by Lemma \ref{lemnew1}, thus $N_\cH\big((e_1\setminus U)\cup \{v_i,v_{r+1}\}\big)\subseteq \cF\cup \{e_r\}$, hence \textbf{(ii)} of the Lemma holds.
    \end{proof}

    Let us return to the proof of the Lemma. Let $e_1\cap U=\{v_{i_0},v_{i_1},\dots,v_{i_s}\}$, then $v_{i_0}=v_2$. We define recursively the sets $Q_j$ the following way. Let $Q_1=e_1\setminus U$ and for $1\le j\le s$, if $(e_{i_j-1}\setminus U')\cap Q_j= \emptyset$, and $e_{i_j-1}\neq U'$ then let $Q_{j+1}=Q_j\cup (e_{i_j-1}\setminus U')$. 
    Otherwise let $Q_{j+1}=Q_j\cup (e_{i_j-1}\setminus U')\cup \{v_{i_j-1}\}$. 
    
    We claim that $|Q_{j+1}|>|Q_j|$. Indeed, when we add $v_{i_j-1}$, it is a new vertex, since we only added vertices outside $U$ and vertices $v_{i_{\ell-1}}$ for some $\ell<j$. Otherwise, we add a set $e_{i_j-1}\setminus U'$ that contains an element not in $Q_j$, unless $e_{i_j-1}\setminus U'=\emptyset$. 
    
    Therefore, $|Q_{s+1}|\ge |Q_1|+s\ge r-1$.  
    Observe that by Claim \ref{claimnew}, $N_\cH(Q_{s+1})\subseteq \cF$. We have $N_\cH(Q_{s+1}\cup \{v_{r+1}\})\subseteq \cF\cup \{e_r\}$, thus we are done unless $|Q_{s+1}|=r-1$.  Now, we repeat the above procedure for the same path, but starting from $v_{r+1}$, and going to the other direction. This way we obtain $R_{t+1}$ with $|R_{t+1}|\ge r-1$ and $N_\cH(R_{t+1}\cup \{v_{1}\})\subseteq \cF\cup \{e_r\}$.
    We are done if $|Q_{s+1}\cup R_{t+1}|\ge r+1$. Otherwise, since $v_{r+1}\not\in Q_{s+1}$ and $v_1\not\in R_{t+1}$, we have that $Q_{s+1}\setminus \{v_1\}=R_{t+1}\setminus \{v_{r+1}\}$. In particular, since $e_1\setminus U\subset Q_{s+1}$ but it cannot be in $R_{t+1}$ (otherwise a Berge-$C_r$ would be found), we have that $e_1\setminus U=\{v_1\}$, and thus $e_1=\{v_1,v_2,\dots, v_r\}$ and $e_r=\{v_2,v_3,\dots, v_{r+1}\}$. 



    If $v_1$ is contained in $e_{i}$ for some $1<i<r$, then the Berge path $v_{i},e_{i-1},\dots, v_2,e_1,v_1, e_{i},v_{i+1}, \\ \dots, e_r,v_{r+1}$ also has that the first edge consists of the first $r$ vertices, thus $e_i=e_1$, a contradiction. 

    

    Therefore, each endpoint of a Berge-$P_r$ has degree 1.
Then each $e_i$ contains a non-defining vertex $w_i$, thus $w_i,e_i,v_i,e_{i-1}, \dots, v_2,e_1,v_{i+1},e_{i+1},\dots,v_r,e_r,v_{r+1}$ shows that $w_i$ is contained only in $e_i$. Thus, the $r$ vertices $v_1,w_2, \dots, w_{r-1},v_{r+1}$ each have degree $1$. Consider now some $v_i$ ($i=2,\dots,r$). If there is no non-defining hyperedge that contains $v_i$, then \textbf{(ii)} holds for $\{v_1,w_2, \dots, w_{r-1},v_{r+1},v_i\}$. 

Assume that there exists a non-defining hyperedge $h$ containing $v_i$. If $h$ contains only defining vertices, then $h\subseteq \{v_1,\dots,v_{r+1}\}$ and $h\neq e_1,e_r$. This implies that $v_1,v_{r+1}\in h$, and thus $h$ forms a Berge-$C_{r+1}$ with  $v_1,e_1,v_2,\dots,e_r,v_{r+1}$, a contradiction. Therefore, 
$h$ contains a non-defining vertex $v\neq w_i$ as $d(w_i)=1$. The Berge path $v,h,v_i,e_{i-1}, \dots, v_2,e_1,v_{i+1}$, $e_{i+1},\dots,v_r,e_r,v_{r+1}$ implies that $v$ is contained only in $h$.
Thus, \textbf{(ii)} holds for $\{v_1,w_2, \dots, w_{r-1},$ $v_{r+1},v\}$, and we are done. 
\end{proof}

Let $\cH$ be an $r$-graph and $U\subseteq V(\cH)$ be a nonempty subset. Let $\cH-U$ denote the $r$-graph obtained from $\cH$ by removing the vertices in $U$ and the hyperedges incident to them. 

Now we are ready to prove our main theorem for $k=r+1$. Recall that it states that if $n=p(r+1)+q$ with $q<r+1$, then $\ex_r(n,\text{Berge-}P_{r+1})=p(r+1)+\binom{q}{r}$.

\begin{proof}[\bf Proof of the case $k=r+1$ of Theorem \ref{main}]
    Let $\cH$ be a counterexample with the smallest number of vertices. Since $\cH$ is Berge-$P_{r+1}$-free, it contains no Berge cycles of length at least $r+2$. 
    Suppose $\cH$ contains a Berge-$C_{r+1}$, say $v_1,e_1,\dots, v_{r+1},e_{r+1}$. Assume that a defining hyperedge, say $e_1$ contains a non-defining vertex $u$. Then $u,e_1,\dots, v_{r+1},e_{r+1},v_1$ is a Berge-$P_{r+1}$, a contradiction. Therefore, these hyperedges form a clique, and any further hyperedge intersecting this clique would also create a Berge-$P_{r+1}$. Deleting this clique we obtain another counterexample on $n-r-1$ vertices, a contradiction. 

    Now we can apply Lemma \ref{lemnew2}. If there are $r+1$ vertices incident to at most $r$ hyperedges, then deleting those vertices we get a smaller counterexample. Otherwise, by Lemma \ref{lemnew2} there is no Berge-$C_r$ in $\cH$. Note that $$e(\cH)>p(r+1)+\binom{q}{r}>p(r-1)+1=\left\lfloor \frac{n}{r+1}\right\rfloor(r-1)+1\geq {\rm{ex}}_r(n,{\rm Berge}{\text -}P_r).$$
    This implies that the longest Berge path in $\cH$ has length $r$, and thus we can apply Lemma \ref{lemnew3}. If there is a set $S$ of size $r-1$ with $|N_\cH(S)|\le 1$ and $0\neq q\neq r-1$, we delete $S$ and find a smaller counterexample. If $q=0$, the statement already follows from the result by Davoodi, Gy\H{o}ri, Methuku and Tompkins (Theorem 3 in \cite{10-10A1}). 


    Now we consider the case $q=r-1$. If $|N_\cH(S)|= 0$, then we delete $S$ and find a smaller counterexample. If $|N_\cH(S)|= 1$, then we consider the subhypergraph $\cH-S$. Note that $|V(\cH-S)|=p(r+1)$ and $e(\cH-S)\geq p(r+1)$. Recall that $\cH$ contains no Berge cycles of length at least $r$. Then 
    $\cH-S$ contains no Berge cycles of length at least $r$. Thus, by Theorem 10 in \cite{GLSZ}, we have \begin{eqnarray*} 
    e(\cH-S)&\leq& \max\left\{\left\lfloor\frac{|V(\cH-S)|-1}{r}\right\rfloor(r-1),|V(\cH-S)|-r+1\right\} \\
    &=& \max\left\{\left\lfloor\frac{p(r+1)-1}{r}\right\rfloor(r-1),p(r+1)-r+1\right\} \\
    &<& p(r+1), 
    \end{eqnarray*}
    which contradicts the fact that $e(\cH-S)\geq p(r+1)$. 
    Finally, if there is a set $S$ of size $r+1$ with $|N_\cH(S)|\le r+1$, we delete $S$ and find a smaller counterexample. This completes the proof.
\end{proof}

\section{Proof of Theorem \ref{main} for $k>r+1$}\label{bigk}

We have mentioned the bound $\ex(n,K_r,F)\le \ex_r(n,\textup{Berge-}F)$. There is also a bound $\ex_r(n,\textup{Berge-}F)\le \ex(n,K_r,F)+\ex(n,F)$ \cite{GP2}. This was strengthened by F\"{u}redi, Kostochka and Luo \cite{FKL} and by Gerbner, Methuku and Palmer \cite{B4} independently in the following way.

A \textit{red-blue graph} is a graph with every edge colored red or blue. We say that a red-blue graph is \textit{$F$-free} if it is $F$-free without considering the colors.
Given a red-blue graph $G$, we denote by $G_r$ the subgraph consisting of the red edges and by $G_b$ the subgraph consisting of the blue edges. Let $g(G):=\cN(K_r,G_b)+e(G_r)$.

\begin{lemma}[Gerbner, Methuku, and Palmer \cite{B4}]\label{lem2.1}
    Given a Berge-$F$-free $n$-vertex $r$-graph $\cH$, there is an $F$-free red-blue $n$-vertex graph $G$ such that $\cH$ has at most $g(G)$ hyperedges.
\end{lemma}

This lemma was used in \cite{B4} to obtain several bounds on Berge-Tur\'an problems, in particular it can be used to prove Theorem \ref{gykl}. However, it cannot prove Theorem \ref{main}, since an $n$-vertex $P_k$-free graph may have $g(G)$ larger than the claimed bound. Indeed, if $q<r$, then a $q$-clique does not contain any $K_r$, but may contain red edges. Furthermore, if $q\le r+1$, then $K_q$ contains more edges than $r$-cliques, thus it makes more sense to color it red. 

We determine the largest value of $g(G)$ among $n$-vertex $P_k$-free red-blue graphs. In particular, this proves Theorem \ref{main} in the case $q\ge r+2$.

We will use the following simple inequalities.

\begin{lemma}\label{lemi}
    \textbf{(i)} If $k\ge r+2$, then $(k-1)/2\le \binom{k-1}{r-1}/r$.

    \textbf{(ii)} If $3\leq r\le \ell\le k-2$, then $(k-\ell-1)/2\le \frac{\binom{k-1}{r-1}-\binom{\ell}{r-1}}{r}$.
\end{lemma}

\begin{proof}[\bf Proof]
    Observe that \textbf{(i)} holds with equality if $k=r+2$ and increasing $k$ increases the right-hand side by more than the left-hand side.

    To prove \textbf{(ii)}, we use induction on $k-\ell-1$. In the case $k-\ell-1=1$, we have that $$\frac{1}{2}\le \frac{r(r-1)}{2r}\le \frac{(k-2)(r-1)}{2r}\le \frac{(k-2)(k-3)\cdot(k-4)\cdots(k-r+1)\cdot2}{2r\cdot(r-2)\cdots3\cdot2}=\frac{\binom{k-2}{r-2}}{r}=\frac{ \binom{k-1}{r-1}-\binom{k-2}{r-1}}{r}.$$
    If $\ell$ decreases by 1, then the left-hand side increases by $\frac{1}{2}$, and the right-hand side increases by $\frac{\binom{\ell}{r-1}-\binom{\ell-1}{r-1}}{r}=\frac{\binom{\ell-1}{r-2}}{r}\geq \frac{r-1}{r}\geq \frac{1}{2}$ since $\ell\ge r$. 
\end{proof}

We will follow the strategy of the proof of Theorem \ref{cc} by Chakraborti and Chen \cite{chch}. In particular, we will use 
Karamata's inequality \cite{kara}.

\begin{lemma}[Karamata's inequality \cite{kara}] Let $f$ be a real valued convex function defined on
$\mathbb{N}$. If $x_1, x_2,\dots , x_N$ and $y_1, y_2,\dots, y_N$ are integers such that
$x_1 \ge x_2\ge \dots \ge x_N$ and $y_1 \ge y_2\ge \dots \ge y_N$,
$x_1 + x_2 + \dots + x_i \ge y_1 + y_2 + \dots + y_i$ for all $i<N$, and $x_1 + x_2 + \dots +x_N = y_1 + y_2 + \dots +y_N$, 
then
$f(x_1) + f(x_2) + \dots +f(x_N) \ge f(y_1) + f(y_2) + \dots +f(y_N)$.    
\end{lemma}

\begin{thm}\label{thm2.4}
    Let $n=pk+q$ with $q<k$. Then for any $P_k$-free red-blue $n$-vertex graph $G$ we have 
    \begin{displaymath}
    g(G)\le p\binom{k}{r}+
    \left\{ \begin{array}{l l}
    \binom{q}{r} & \textrm{if\/ $q\ge r+2$},\\
    \binom{q}{2} & \textrm{if\/ $q\le r+2$}.\\
    \end{array}
    \right.
    \end{displaymath}
    Furthermore, this upper bound can be achieved. 
\end{thm}

\begin{proof}[\bf Proof]
    The lower bounds are given by the graph consisting of $p$ mono-blue copies of $K_k$ and a copy of $K_q$ that is mono-blue if $q\ge r+2$ and mono-red otherwise. 

    To prove the upper bound, we follow the strategy of Chakraborti and Chen \cite{chch}. For a vertex $v$, let $d_b(v)$ denote the number of blue edges incident to $v$ and $d_r(v)$ denote the number of red edges incident to $v$. Let $k(v)$ denote the number of blue copies of $K_r$ containing $v$.
    The \textit{contribution} of $v$ is $c(v):=\frac{k(v)}{r}+\frac{d_r(v)}{2}$. Clearly, $\sum_{v\in V(G)} c(v)=g(G)$. 

    For any vertex $v$ with $d_b(v)=a(v)(k-1)+b(v)$ we have that $k(v)\le a(v)\binom{k-1}{r-1}+\binom{b(v)}{r-1}$ using Theorem \ref{cc}, since the neighborhood of $v$ is $P_{k-1}$-free. Therefore, 
    \begin{equation}\label{eq1}
    g(G)\le\sum_{v\in V(G)} \bigg[a(v)\frac{\binom{k-1}{r-1}}{r}+\frac{\binom{b(v)}{r-1}}{r}+\frac{d_r(v)}{2}\bigg].\end{equation}

The proof of Theorem \ref{cc} in \cite{chch} goes the following way. If $d_r(v)=0$, then the right-hand side of the above equation corresponds to a sequence $\{y_i\}_{i=1}^N$ consisting of $\sum_{v\in V(G)} a(v)$ copies of $k-1$ and $b(v)$ for each $v$.
They construct a sequence $x_i$ such that $x_i=k-1$ for $i\le pk$, $x_i=q-1$ for all $pk<i\le pk+q$, and $x_i=0$ for $i>pk+q$. Then they show that the conditions of Karamata's inequality hold for $x_i$ and $y_i$, which completes their proof (using that $\binom{x}{r-1}/r$ is a convex function).





Let $y:=\sum_{v\in V(G)} d_r(v)$.
 We will now increase the right-hand side of (\ref{eq1})
 the following way. We will obtain a quantity $\sum_{i=1}^m \frac{\binom{z_i}{r-1}}{r}+z/2$ for some $m$ such that the following three conditions hold for the sequence $z,z_1,z_2,\dots, z_m$.
 
 1, $k-1\ge z_1\ge z_2\ge\dots\ge z_m$, 
 
 2, $\sum_{i=1}^{pk+1} \binom{z_i}{r-1}\le pk\binom{k-1}{r-1}+\binom{q-1}{r-1}$ and 
 
 3, $z+\sum_{i=1}^m z_i=2e(G)$.

 First we show that these conditions hold originally, i.e., for the sequence $y,y_1,\dots, y_N$.
 Observe that $\sum_{v\in V(G)} [a(v)(k-1)+b(v)+d_r(v)]=\sum_{v\in V(G)}d(v)=2e(G)$, thus the third condition holds for $y_i$. The first condition holds for $y_i$ by definition, while the second condition holds because $\sum_{i=1}^{pk+1} \binom{y_i}{r-1}\le\sum_{i=1}^{pk+1} \binom{x_i}{r-1}= pk\binom{k-1}{r-1}+\binom{q-1}{r-1}$.
 

Recall that the conditions of Karamata's inequality hold for $x_i$ and $y_i$; in particular, there are at most $pk$ elements in $\{y_i\}_{i=1}^N$ equal to $k-1$. If there are fewer than $pk$ such elements, consider the next two elements $b(v_1), b(v_2)$. If both are larger than 0, we combine them in the following sense. If $b(v_1)+b(v_2)\le k-1$, we replace them by $b(v_1)+b(v_2)$ and 0. 
If $b(v_1)+b(v_2)> k-1$, we replace them by $k-1$ and $b(v_1)+b(v_2)-k+1$. After reorganizing the sequence, the first and third conditions obviously still hold. The second condition holds since we changed $y_i$ and $y_{i+1}$ for some $i\le pk$, and then only the $i$th element can increase.
We repeat this procedure as long as there are fewer than $pk$ elements in our sequence that are equal to $k-1$. Then the required conditions hold at the end.

This procedure can end in two ways: either only at most one $b(v)>0$ remains, or we have $pk$ elements equal to $k-1$. We let $B:=\sum_{i=1}^m \frac{\binom{z_i}{r-1}}{r}$ and $R:=z/2$. Note that $B$ is an upper bound on the total contribution of the blue cliques and $R$ is an upper bound on the total contribution of the red edges, but they do not correspond to blue cliques or red edges.

CASE 1. We arrived at a sequence $\{z_i\}$ that has $pk-t$ elements equal to $k-1$ and at most one element equal to $k-s$ for some integers $t>0$ and $s>1$. Then we have that $B$ is at most $(pk-t)\binom{k-1}{r-1}/r+\binom{k-s}{r-1}/r$. Note that the underlying graph of $G$ is $P_k$-free. By the third condition, $z=2e(G)-\sum_{i=1}^m z_i\leq 2\cdot\ex(n,P_k)-\sum_{i=1}^m z_i=(t-1)(k-1)+s-1+q(q-1)$, hence $R$ is at most $(t-1)(k-1)/2+(s-1)/2+q(q-1)/2$.




We have that $(t-1)(k-1)/2\le (t-1)\binom{k-1}{r-1}/r$ and $(s-1)/2\le \frac{\binom{k-1}{r-1}-\binom{k-s}{r-1}}{r}$ by Lemma \ref{lemi}. Therefore, our upper bound on the sum of contributions is at most $pk\binom{k-1}{r-1}/r+\binom{q}{2}$, which completes the proof.

CASE 2. We arrived at a sequence $\{z_i\}$ that has $pk$ elements equal to $k-1$ and some other elements less than $q$. We have that $z':=\sum_{i=pk+1}^m z_i=2e(G)-pk(k-1)-z$. 
Let $z'=a(q-1)+b$ with $b<q-1$, then we further increase the right-hand side of (\ref{eq1}) if we replace the $z_i$'s with $i>pk$ by $a$ elements equal to $q-1$ and an element equal to $b$. 
We have that $z\le q(q-1)-z'$. 

If $q\ge r+2$ and $b\ge r-1$, then $\frac{a\binom{q-1}{r-1}+\binom{b}{r-1}}{r}+\frac{z}{2}\le\frac{(a+1)\binom{q-1}{r-1}}{r}+\frac{z-q+1+b}{2}\le \frac{(a+1)\binom{q-1}{r-1}}{r}+\frac{(q-a-1)(q-1)}{2}\le \frac{q\binom{q-1}{r-1}}{r}$. Here, the first inequality holds because by decreasing $z$ and increasing $b$ by 1, the value of the expression increases by at least $\binom{b}{r-2}/r-1/2$.
The second inequality holds because $z\le q(q-1)-z'$, and the third inequality holds because the expression increases by $(q-a-1)(\frac{\binom{q-1}{r-1}}{r}-(q-1)/2)$, which is non-negative by \textbf{(i)} of Lemma \ref{lemi}.
 If $q\ge r+2$ and $b<r-1$, then observe that $z\le (q-a)(q-1)$, thus $z/2\le (q-a)(q-1)/2\le (q-a)\binom{q-1}{r-1}/r$. Therefore, $\frac{a\binom{q-1}{r-1}+\binom{b}{r-1}}{r}+\frac{z}{2}\le \frac{q\binom{q-1}{r-1}}{r}$, using that $\binom{b}{r-1}=0$.
This implies that our upper bound on $g(G)$ is at most $pk\binom{k-1}{r-1}/r+\binom{q}{r}$, which completes the proof.

If $q\le r+1$, then $b$ is at most $r-1$, thus $\binom{b}{r-1}$ is either 0 or 1. 
We increase $\frac{a\binom{q-1}{r-1}}{r}$ to $a(q-1)/2$ and if $b=r-1$, then we increase $\frac{\binom{b}{r-1}}{r}$ to $b/2$. This way we increased $\frac{a\binom{q-1}{r-1}+\binom{b}{r-1}}{r}+\frac{z}{2}$ to $(z+z')/2\le q(q-1)/2$. Therefore, our upper bound on the sum of contributions is at most $pk\binom{k-1}{r-1}/r+\binom{q}{2}$, which completes the proof.
\end{proof}


We say an $r$-graph $\cH$ is \textit{hamiltonian-connected} if it contains a hamiltonian Berge path between any pair of its vertices.  Let $\delta(\cH)$ denote the minimum degree of vertices in $\cH$. For $S\subseteq V(\cH)$, let $\cH[S]$ denote the subhypergraph of $\cH$ induced by $S$. 

\begin{thm}[Kostochka, Luo and McCourt \cite{klm}]\label{thm3.1}
Let $n \ge r \ge 3$. Suppose $\cH$ is an $n$-vertex $r$-graph such that
    \textbf{(i)} $r \le \frac{n}{2}$ and $\delta(\cH) \ge \binom{\lfloor n/2 \rfloor}{r-1} + 1$, or
    \textbf{(ii)} $n - 1 \ge r > \frac{n}{2} \ge 3$ and $\delta(\cH) \ge r - 1$, or
    \textbf{(iii)} $r = 3$, $n = 5$ and $\delta(\cH) \ge 3$.
Then $\cH$ is hamiltonian-connected.
\end{thm}

Given an $r$-graph $\cH$ with the longest Berge path of length $\ell>r$, we say that a set $S\subseteq V(\cH)$ is \textit{good} if $|N_\cH(S)|\le |S|\binom{\ell}{r-1}/r$, and a 2-set $S$ is \textit{very good} if $|N_\cH(S)|\le \binom{\ell}{r-1}/r+\binom{\ell-1}{r-1}/r$. The proof of the bound ${\rm{ex}}_r(n,{\rm Berge}{\text -}P_k)\leq \frac{n}{k}\binom{k}{r}$ in \cite{B6} goes by showing a good set and applying induction. We need a good set of order exactly $k$ to apply induction.
We will use the following lemma.


\begin{lemma}\label{lemmi}
Assume that the longest Berge path in an 
$r$-graph $\cH$ 
has length $\ell>r$. Then at least one of the following holds.

\begin{itemize}
    \item There is a good set of size 1 and a good set of size 2.

    \item There is a good set of size 2 and a good set of size 3.

    \item The defining vertices of a Berge path of length $\ell$ form a connected component that either contains a Berge cycle of length $\ell+1$, 
    or contains a Berge cycle of length $\ell$ and the remaining vertex of the Berge path has degree 1 in $\cH$.
\end{itemize}

      Moreover, if we found a  good set $S$ of size two this way, and there are at most $\ell-1$ vertices in addition to $S$ in $\cH$ and $\ell>r+1$, then the number of hyperedges is at most $\binom{\ell-1}{r}+\binom{\ell-1}{r-1}/r+\binom{\ell}{r-1}/r$. 

      

\end{lemma}

Note that the moreover part follows if we find a very good set of size 2.
Note also that in \cite{B6}, a good set of unspecified size was found.
We borrow a lot from the proof in \cite{B6}, but there are two small mistakes in the argument there that we also have to correct.
We remark that in the case $k\ge r+3$, one can find a good set of size 1 with a little bit more work. That would be enough for our purposes, but we also have to deal with the case $k=r+2$.

\begin{proof}[\bf Proof] 
Consider a Berge path $v_1,e_1,v_2,\dots,e_\ell, v_{\ell+1}$ of length $\ell$. First, assume there is a Berge cycle with defining vertices $v_1,\dots,v_{\ell+1}$. Without loss of generality, the defining hyperedges are $e_1,\dots, e_\ell$
and a hyperedge $e_{\ell +1}$ containing $v_{\ell+1}$ and $v_1$. 
If there is a hyperedge $h$ that is not a defining hyperedge and contains $v_i$ and a vertex $v$ that is not a defining vertex, then $v,h,v_i,e_i,v_{i+1},\dots, v_{i-1}$ is a Berge path of length $\ell+1$, a contradiction. If a defining hyperedge $e_i$ contains a non-defining vertex $v$, then $v,e_i,v_{i+1},\dots, v_{i-1}, e_{i-1}, v_i$ is a Berge path of length $\ell+1$, a contradiction. Therefore, the third bullet point holds for this cycle. Hence, we are done unless 
there is no Berge-$C_{\ell+1}$ in $\cH$ on the vertices $v_1,v_2,\dots, v_{\ell+1}$.


    Let $\cH'$ denote the set of hyperedges in $\cH$ that are not the defining hyperedges of our Berge path. Note that the moreover part follows if the good set of size 2 is very good.

    CASE 1. No hyperedge of $\cH'$ contains $v_1$ or $v_{\ell+1}$.
   Then $N_{\cH}(\{v_1,v_{\ell+1}\})\leq \ell$. 
   Observe that since $\ell>r$, we have $\binom{\ell}{r-1}/r\ge \binom{\ell}{\ell-2}/(\ell-1)=\ell/2$. Therefore, $\{v_1,v_{\ell+1}\}$ is a good set. Moreover, $\{v_1,v_{\ell+1}\}$ is a very good set
since $\binom{\ell}{r-1}/r+\binom{\ell-1}{r-1}/r\ge \ell$ when $\ell > r+1$.

We have found a good set of size $2$. Next, we will find a good set of size $1$ or $3$.
   If any of $v_1$ and $v_{\ell+1}$ is contained in at most $\lfloor\frac{\ell}{2}\rfloor$ defining hyperedges, then it will be a good set of size 1, and we are done. Now assume that each of $v_1$ and $v_{\ell+1}$ is contained in at least $\lfloor\frac{\ell}{2}\rfloor+1$ defining hyperedges. 
   Take a hyperedge $e_i$ that contains $v_{\ell+1}$. Then $v_1,e_1,v_2, \ldots, v_i,e_i, v_{\ell+1}, e_\ell, v_\ell, \ldots, v_{i+1}$ is a Berge-$P_\ell$ with the same defining hyperedges. If $v_{i+1}$ is contained only in the defining hyperedges, then $\{v_1,v_{\ell+1},v_{i+1}\}$ is a good set of size three, and we are done. We claim that any non-defining hyperedge that contains $v_{i+1}$ must only contain defining vertices, otherwise we can get a Berge-$P_{\ell+1}$. Let $A_{i+1}$ be the set of defining vertices that are in a non-defining hyperedge with $v_{i+1}$. Note that if some $v_j$ is in $A_{i+1}$, then $v_1$ cannot be in $e_j$, otherwise we get a Berge-$C_{\ell+1}$. Thus, $|A_{i+1}| \leq \ell - (\lfloor\frac{\ell}{2}\rfloor+1)= \lceil \frac{\ell}{2}\rceil -1$. So, there are at most $\binom{\lceil \frac{\ell}{2}\rceil -1}{r-1}$ non-defining hyperedges containing $v_{i+1}$. Therefore, $\{v_1, v_{\ell+1}, v_{i+1}\}$ is incident to at most $\ell+\binom{\lceil \frac{\ell}{2}\rceil -1}{r-1}\leq \ell+\binom{\ell}{r-1}/r$ hyperedges, hence forms a good set of size 3.


CASE 2. 
For each Berge path of length $\ell$, at least one of the endvertices is contained in a non-defining hyperedge. Note that such non-defining hyperedges contain only the defining vertices, otherwise we get a Berge-$P_{\ell+1}$. 
     Let $U$ be the defining vertex set of the Berge path $v_1,e_1, \ldots, v_l,e_\ell,v_{\ell+1}$. Let $A_1$ denote the set of vertices in $U\setminus \{v_1\}$ that are contained in a hyperedge of $\cH'$ together with $v_1$ and $A_{\ell+1}$ denote the set of vertices in $U\setminus \{v_{\ell+1}\}$ that are contained in a hyperedge of $\cH'$ together with $v_{\ell+1}$. 
    A hyperedge in $\cH'$ containing $v_1$ and $v_{\ell+1}$ would create a Berge cycle of length $\ell+1$. Meanwhile, if $v_i\in A_{\ell+1}$ and $v_{i+1}\in A_1$, then $v_i,e_{i-1},v_{i-1},\dots,e_1,v_1,h_1,v_{i+1},e_{i+1},\dots,e_\ell,v_{\ell+1},h_2,v_i$ is a Berge cycle of length $\ell+1$, where $v_1,v_{i+1}\in h_1\in \cH'$ and $v_i,v_{\ell+1}\in h_2\in \cH'$. Therefore, $|A_1|$ defining vertices are forbidden from $A_{\ell+1}$. Note that $v_1$ and $v_{\ell+1}$ are also not in $A_{\ell+1}$.
   The proof in \cite{B6} concludes that $|A_1|+|A_{\ell+1}|\le \ell-1$, but this is incorrect, since $v_1$ may be one of the $|A_1|$ forbidden vertices, thus we only have $|A_1|+|A_{\ell+1}|\le \ell$.


    Clearly, the above imply that at least one of $|A_1|\ge r-1$ and $|A_{\ell+1}|\ge r-1$ holds, and there are at most $\binom{|A_1|}{r-1}$ hyperedges in $\cH'$ that contain $v_1$ and analogously at most $\binom{|A_{\ell+1}|}{r-1}$ hyperedges in $\cH'$ that contain $v_{\ell+1}$. 
  Observe that if $v_i\in A_1$, and $e_{i-1}$ contains $v_{\ell+1}$, then $v_i,e,v_1,e_1,\dots, v_{i-1},e_{i-1}, v_{\ell+1},e_\ell, v_\ell,\dots, v_i$ is a Berge cycle of length $\ell+1$, where $\{v_1,v_i\}\subseteq e\in \cH'$. 
  Therefore, $v_{\ell+1}$ is contained in at most $\ell-|A_1|$ defining hyperedges and analogously $v_1$ is contained in at most $\ell-|A_{\ell+1}|$ defining hyperedges. Thus, if the smaller of $A_1$ and $A_{\ell+1}$ contains $i$ elements, then the degree of an endpoint is at most $\binom{i}{r-1}+\ell-i$. If $i\ge r-1$, then increasing $i$ increases this number, thus its maximum is $\binom{\lfloor\ell/2\rfloor}{r-1}+\lceil \ell/2 \rceil$.
  If $i<r-1$, then $i=0$, say, $A_{\ell+1}=\emptyset$. Then $|A_1|\ge r-1$, thus the degree of $v_{\ell+1}$ is at most $\ell-r+1$. Moreover, if $i\ge r-1$, then $\ell\ge 2r-2$.






    We have obtained that there is a vertex, say $v_{\ell+1}$ of degree at most $x$, where $x=\ell-r+1$ if $\ell<2r-2$ and $x=\max\{\binom{\lfloor\ell/2\rfloor}{r-1}+\lceil\ell/2\rceil,\ell-r+1\}$ if $\ell\ge 2r-2$. 

    

    
    We need to show that $x\le y:=\binom{\ell}{r-1}/r$. Let us remark that \cite{B6} arrived at a similar inequality and left checking the details to the reader, but their inequality does not hold for $r=3$, $k=6$.

    We claim that if $r\ge 4$ and $\ell$ increases by 1, then $y-x$ does not decrease. Clearly, $y$ increases by $\binom{\ell}{r-2}/r$, while $x$ increases by at most $\binom{\lfloor\ell/2\rfloor}{r-2}+1$. If $\ell<2r-2$, then $x$ increases by 1 and we are done. Otherwise, we have $\binom{\lfloor\ell/2\rfloor}{r-2}\le\binom{\ell/2}{r-2}=\frac{\ell(\ell-2)\dots (\ell-2r+6)}{2^{r-2}(r-2)!}<\binom{\ell}{r-2}/r$.
    Therefore, it is enough to check that $x\le y$ holds for the smallest possible $\ell$, which is $\ell=2r-2$. In that case, we have to show that $r\le \binom{2r-2}{r-1}/r$. This can be easily checked.




    It is left to deal with the case $r=3$. In that case, we have that $y=\ell(\ell-1)/6$ (in which case $\ell\ge 4$), while for odd $\ell$, either $x\le (\ell-1)(\ell-3)/8+(\ell+1)/2$, or $x=\ell-2$. Then one can easily check that $x\le y$ if $\ell\ge 7$. 
    
    If $\ell=5$, then 
    $y=10/3$, while it is possible that $x=4$. However, in this case we have that $|A_1|,|A_6|\ge 2$. Recall that if $v_i\in A_1$, then $v_{i-1}\not\in A_6$ and $e_{i-1}$ does not contain $v_6$. Analogously, if $v_i\in A_6$, then $v_{i+1}\not\in A_1$ and $e_i$ does not contain $v_1$. 
    
    First, assume that $v_2\not\in A_1$. Then we either have that $v_2\in A_6$ (in which case $A_1=\{v_4,v_5\}$) or $v_2\not \in A_6$. In the first case, we have that $A_6=\{v_2,v_5\}$, $v_1$ is contained in $e_1,e_3,e_4$ and $v_6$ is contained in $e_1,e_2,e_5$. Then $v_1v_4v_5$ is a defining hyperedge, thus $v_1$ is not contained in any non-defining hyperedge, hence $v_1$ is contained in $3$ hyperedges, a contradiction. In the second case, we have that $A_1\cup A_6\subseteq \{v_3,v_4,v_5\}$, which is possible only if $A_1=A_6=\{v_3,v_5\}$ (and $v_1$ is contained in $e_1,e_2,e_4$, $v_6$ is contained in $e_1,e_3,e_5$) or $A_1=\{v_3,v_4\}$ and $A_6=\{v_4,v_5\}$ (and $v_1$ is contained in $e_1,e_2,e_3$, $v_6$ is contained in $e_1,e_4,e_5$). In the last case, $v_1v_3v_4$ is a defining hyperedge, thus $v_1$ is not contained in any non-defining hyperedge, a contradiction as above. 
    If $A_1=A_6=\{v_3,v_5\}$, then the hyperedges are $v_1v_2v_6(e_1)$, $v_1v_2v_3(e_2)$, $v_3v_4v_6(e_3)$, $v_1v_4v_5(e_4)$, $v_5v_6x(e_5)$, $v_1v_3v_5$, $v_3v_5v_6$, for some $x$.
    Then $x$ must be a defining vertex, otherwise we use $v_3v_5v_6$ for the edge $v_5v_6$ in the original Berge path, and then continue with $v_6xv_5$ to $x$, obtaining a Berge path of length 6, a contradiction. 
    Then $x\in\{v_1,v_2,v_4\}$. If $x\in\{v_1,v_2\}$, then again, we use $v_3v_5v_6$ for the edge $v_5v_6$ in the original Berge path, and then the non-defining hyperedge containing $v_6$ either contains $v_1$, forming a Berge cycle of length 6, a contradiction, or contains $v_2$, arriving to the previous case. If $x=v_4$, then we can find a Berge-$C_6$ ($v_5\xrightarrow{v_3v_5v_6} v_6\xrightarrow{e_5} v_4\xrightarrow{e_3} v_3\xrightarrow{e_2} v_2\xrightarrow{e_1} v_1\xrightarrow{v_1v_3v_5} v_5$), a contradiction. 

    We have obtained that $v_2\in A_1$, and then analogously $v_5\in A_6$. We can replace $e_1$ by the non-defining hyperedge $h$ containing $v_1$ and $v_2$. Then the third vertex of $e_1$ must be a defining vertex $v_i$. Let $A_1'=A_1\cup \{v_i\}$. If $v_i\not\in A_1$, then $|A_1'| \geq 3$. If $v_i\in A_1$, then  there is a non-defining hyperedge $\{v_1, v_i, x\}$, where $x\neq v_2$ since $\{v_1, v_i, x\}\neq e_1$. Therefore, $|A_1'| \geq 3$ in this case as well.
Analogously, let $A_6'=A_6\cup\{v_j\}$, where $v_j$ is the vertex of $e_5$ different from $v_5$ and $v_6$. Then we have that $|A_6'| \geq 3$. Observe that $v_q\in A_6'$ and $v_{q+1}\in A_1'$ is impossible, by the same argument that works for $A_6$ and $A_1$. This shows that $|A_1'|+|A_6'|\le 5$, a contradiction.

    For $r=3$ and even $\ell$, we claim that there is a vertex of degree at most $\ell(\ell-2)/8+\ell/2-1$. This clearly holds if $x=\ell-2$, thus we can assume that $x=\ell(\ell-2)/8+\ell/2$. We are done unless $v_{\ell+1}$ has degree exactly $x$. In that case, we have $|A_1|=|A_{\ell+1}|=\ell/2$ and each defining hyperedge contains exactly one of $v_1$ and $v_{\ell+1}$. Indeed, this second condition has been established for $e_1$ and $e_\ell$, the other defining hyperedges cannot contain both $v_1$ and $v_{\ell+1}$ since $r=3$, and hence each of $v_1$ and $v_{\ell+1}$ is contained in $\ell/2$ defining hyperedges. Moreover, each subset of $A_1$ of size $r-1=2$ forms a non-defining hyperedge when extended with $v_1$.
    
   We know that $v_i\in A_1$ for some $i>2$. Then, as we have observed, $e_{i-1}$ cannot contain $v_{\ell+1}$, thus contains $v_1$. Consider now the Berge path $v_2,e_2,v_3,\dots, v_{i-1},e_{i-1},v_1,h,v_i,e_i, \dots, v_{\ell+1}$, where $h$ is a non-defining hyperedge of our Berge path that contains $v_1$ and $v_i$. This Berge path has length $\ell$, thus we could start with this Berge path instead of our Berge path. If we do not find a vertex of degree at most $\ell(\ell-2)/8+\ell/2-1$ there, then each defining hyperedge of this Berge path contains one of the endpoints $v_2$ or $v_{\ell+1}$. More precisely, $\ell/2$ defining hyperedges of the new path contain $v_2$, but we also know that only $e_2$ and $h$ may contain $v_2$, thus we are done if $\ell>4$. If $\ell=4$, we are done unless $h$ consists of $v_1,v_2,v_i$. In that case $A_1=\{v_2,v_i\}$. 
   


   If $i=3$, then $h=\{v_1,v_2,v_3\}$, and hence $A_1=\{v_2,v_3\}$. Since $v_3 \in A_1$, $v_{l+1} \notin e_2$, but then $e_2$ must contain $v_1$. Thus, $e_2=\{v_1,v_2,v_3\}=h$, a contradiction.

   
    If $i=4$, then $A_{\ell+1}=A_1=\{v_2,v_4\}$. We claim that $e_3$ cannot contain any of $v_1$ and $v_{\ell+1}$. Otherwise, we can find a Berge-$C_{\ell+1}$ ($v_3\rightarrow v_2\rightarrow v_5\rightarrow v_4\rightarrow v_1\rightarrow v_3$) if $v_1\in e_3$, and a Berge-$C_{\ell+1}$ ($v_4\rightarrow v_5\rightarrow v_3\rightarrow v_2\rightarrow v_1\rightarrow v_4$) if $v_5\in e_3$, a contradiction. Recall that $e_{3}$ does not contain $v_{\ell+1}$ but contains $v_1$ (since $v_4\in A_1$), which is a contradiction. 
    If $i = 5$, it is easy to see that there is a Berge-$C_{\ell+1}$, a contradiction. Thus, we are done, since $\ell(\ell-2)/8+\ell/2-1\le y$ for even $\ell$. 
   
    We have found a good set of size 1. More precisely, we showed that at least one end vertex of any Berge path of length $\ell$ forms a good set. 
    Without loss of generality, suppose that in the Berge path $v_1,e_1,v_2,\dots,e_\ell, v_{\ell+1}$, the vertex $v_{\ell+1}$ forms a good set of size 1. Next, we will find a good set of size 2.

    Assume now that $v_{\ell+1}$ is in $e_i$ for some $i<\ell$. Then consider the Berge path $v_1,e_1,\dots, v_i$, $e_i,v_{\ell+1},e_\ell, v_\ell, \dots, e_{i+1},v_{i+1}$. It has length $\ell$, thus at least one of $v_1$ and $v_{i+1}$ has degree at most $\binom{\ell}{r-1}/r$ and forms a good set together with $v_{\ell+1}$. If $v_{\ell+1}$ is contained in a non-defining hyperedge $h$, then that contains only defining vertices. Let $v_i\in h$, then $v_1,e_1,\dots,e_{i-1},v_i,h,v_{\ell+1},e_\ell$, $v_\ell,\dots, v_{i+1}$ is a Berge path of length $\ell$, thus at least one of its endpoints  has degree at most $\binom{\ell}{r-1}/r$ and forms a good set together with $v_{\ell+1}$. 


    Assume now that $v_{\ell+1}$ is contained in only one hyperedge, i.e., $e_\ell$. If $e_\ell$ contains a non-defining vertex, then that can replace $v_{\ell+1}$, and by the above reasoning, we find either a good set of size 2, or a vertex of degree 1, which forms a good set with $v_{\ell+1}$. 
    
    Next, suppose that $e_\ell$ contains no non-defining vertices. 
    Let us delete now $v_{\ell+1}$ and the single hyperedge $e_\ell$ containing it and repeat the argument of the proof so far. If there is in the resulting hypergraph $\cH_0$ a Berge-$P_\ell$ and a Berge cycle of length $\ell+1$ containing it, we are done. 
    If there is a Berge path of length $\ell$, and CASE 1 holds inside $\cH_0$, then we find a good set $\{v,v'\}$ in $\cH_0$. Note that $v$ and $v'$ are the two endpoints of this Berge path of length $\ell$. This implies that $v,v'\notin e_\ell$, otherwise we would obtain a Berge-$P_{\ell+1}$ in $\cH$ together with the hyperedge $e_\ell$ and the vertex $v_{\ell+1}$, a contradiction. Hence, $\{v,v'\}$ is also a good set in $\cH$, and we are done. 
    If there is a Berge path of length $\ell$, and CASE 2 holds inside $\cH_0$, then we find at least one endvertex $v$ that forms a good set. 
    Clearly, $\{v_{\ell+1},v\}$ is a good set as $\binom{\ell}{r-1}/r+1\leq 2\binom{\ell}{r-1}/r$. 
    

    If there is no Berge path of length $\ell$ in $\cH_0$ and $\ell-1>r$, then we repeat this argument with the Berge path $v_1,e_1, \dots, e_{\ell-1},v_\ell$ in $\cH_0$. If we get a Berge-$C_\ell$, then together with $v_{\ell+1}$, the third bullet point holds. If CASE 1 holds in $\cH_0$, then we find the required good sets $S$ with $|N_{\cH_0}(S)|\le |S|\binom{\ell-1}{r-1}/r\le |S|\binom{\ell}{r-1}/r-1$, thus $|N_{\cH}(S)|\le |S|\binom{\ell}{r-1}/r$. If CASE 2 holds in $\cH_0$, then we find a good vertex, which forms a very good set with $v_{\ell+1}$.
    If we find that $v_1, \dots,v_\ell$ form a connected component in $\cH_0$, then there are two possibilities.
    
    
    If $v_1, \dots,v_\ell$ form a connected component that contains a vertex of degree 1, then that vertex with $v_{\ell+1}$ forms a good set. If $v_1, \dots,v_\ell$ form a connected component in $\cH_0$ which contains a Berge cycle with length $\ell$, then with $v_{\ell+1}$ they form a connected component in $\cH$ with a Berge cycle of length $\ell$ and a vertex of degree 1. In each of the above cases, we found the desired configuration. Finally, assume that $\ell-1=r$. If there is a Berge cycle of length $\ell$, then the third bullet point holds and we are done. Otherwise, we can apply Lemma \ref{lemnew1} to find a vertex $v$ only in defining hyperedges. Then $|N_{\cH}(\{v,v_{\ell+1}\})|\le \ell$ and we are done.

    
   
   It is left to prove the moreover part. Recall that in CASE 1, we have proved the moreover part (we found a very good 2-set), thus we can assume we are in CASE 2, in particular, $v_{\ell+1}$ forms a good set.
   Assume that $\cH$ has exactly $\ell+1$ vertices and delete $v_{\ell+1}$. Then, in the remaining hypergraph $\cH_0$ the longest path has length $\ell -1$. If there is a Berge cycle of length $\ell$ on the $\ell$ vertices in $\cH_0$, then there are two cases. If $v_{\ell+1}$ has degree 1 in $\cH$, then we do not need to show the moreover part, because the third bullet point of the statement holds. If $v_{\ell+1}$ is in at least two hyperedges, and $\cH_0$ is hamiltonian-connected, then we pick two distinct vertices $v_i$ from one of the hyperedges containing $v_{\ell+1}$ and $v_j$ from the other. Then the hamiltonian Berge path connecting $v_i$ and $v_j$ in $\cH_0$ extends to a Berge cycle of length $\ell+1$, a case we have dealt with in the first paragraph of the proof. If $\cH_0$ is not hamiltonian-connected, then it contains a vertex $v$ of degree at most $\binom{\lfloor \ell/2\rfloor}{r-1}$ or $r-2$ or $2$ by Theorem \ref{thm3.1}. Then in each case, $\{v,v_{\ell+1}\}$ is a very good set.
    
    Assume now that there is no cycle of length $\ell$ in $\cH_0$.
    If in $\cH_0$ each Berge path of length $\ell-1$ has an endpoint contained in some non-defining hyperedges, then we can find a good vertex by repeating the above argument. The good vertex there is contained in at most $\binom{\ell-1}{r-1}/r$ hyperedges. Recall that $d_{\cH}(v_{\ell+1})\leq \binom{\ell}{r-1}/r$. Then $e(\cH)\leq \binom{\ell}{r-1}/r+\binom{\ell-1}{r-1}/r+\binom{\ell-1}{r}$, and thus we are done. 
    If there is a Berge path of length $\ell-1$ in the remaining graph such that none of its end vertices is contained in a non-defining hyperedge, then there are at most $\ell-1$ hyperedges containing those two vertices and at most $\binom{\ell-2}{r}$ hyperedges avoiding those two vertices in $\cH_0$. Therefore, there are at most $\binom{\ell-2}{r}+\ell-1+\binom{\ell}{r-1}/r$ hyperedges in $\cH$. This is at most the desired bound if and only if $\ell-1\le \binom{\ell-2}{r-1}+\binom{\ell-1}{r-1}/r$. Since $\ell>r+1$, we have $\ell-2>r-1$, thus the first term of the right-hand side is at least $\ell-2$. The second term of the right-hand side is at least $(\ell-1)/r\ge 1$, thus we are done. 
\end{proof}

Now we are ready to prove Theorem \ref{main} for $k>r+1$. Recall that it states that if $n=pk+q$ with $q<k$, then $\ex_r(n,\textup{Berge-}P_k)=p\binom{k}{r}+\binom{q}{r}$. 

\begin{proof}[\bf Proof of Theorem \ref{main}]
For the lower bound, we partition the $n$ vertices into $p$ sets of size $k$ and a single $q$-set. In each set, take all possible subsets of size $r$ as hyperedges of the hypergraph. The resulting $r$-graph has exactly $p\binom{k}{r}+\binom{q}{r}$ hyperedges and clearly contains no copy of any Berge-$P_k$. 

Let us continue with the upper bound. If $q=0$, we are done by Theorem \ref{gykl}.
If $q\ge r+2$ or $q=1$, we are done by Lemma \ref{lem2.1} and Theorem \ref{thm2.4}. From now on, we assume that $2\le q\le r+1$.

Assume that the statement does not hold and consider a counterexample on the least number of vertices. It is an $n$-vertex Berge-$P_k$-free $r$-graph $\cH$ for some $n=pk+q$. The statement is trivial for $n\le k$, thus we can assume $n>k$. We can also assume that $\cH$ is connected, otherwise at least one of its connected components is a smaller counterexample. If $\cH$ contains a good $k$-set, we delete the good $k$-set to obtain a hypergraph $\cH^*$, then $\cH^*$ is not a counterexample, thus has at most $(p-1)\binom{k}{r}+\binom{q}{r}$ hyperedges. We have deleted at most $k\binom{k-1}{r-1}/r=\binom{k}{r}$ hyperedges, thus $\cH$ contains at most $p\binom{k}{r}+\binom{q}{r}$ hyperedges, a contradiction.

If $\cH$ contains
good sets of size 1 and 2, or good sets of size 2 and 3, then we delete such a good set with parity equal to the parity of $k$ to obtain $\cH_1$. 
Then if $\cH_1$ has a good set of size 2 that satisfies the moreover part of Lemma \ref{lemmi}, we delete such a good set to obtain $\cH_2$, and so on.
If we can delete exactly $k$ vertices this way, then we found a good $k$-set, a contradiction. 

Therefore, there are at most $k-2$ vertices that we can delete this way, let $S$ denote their set and $\cH'$ denote the remaining hypergraph. 


\begin{clm}
Each component of $\cH'$ has order at most $k$ and contains a Berge path of length more than $r$, except at most one isolated vertex.   
\end{clm}

\begin{proof}[\bf Proof of Claim]
Let us assume indirectly that there is a component $C$ of $\cH'$ that is $P_{r+1}$-free.
    If there is a Berge-$C_{r+1}$ in $C$, then no hyperedge of $\cH'$ contains a defining vertex from the Berge-$C_{r+1}$ and a vertex that is not a defining vertex. Therefore, $C$ is of order $r+1$, hence clearly contains a good set of size 2, since $C$ has at most $\binom{r+1}{r}$ hyperedges. 
    If there is a Berge-$C_r$ in $C$, then Lemma \ref{lemnew2} gives us a good set of size 2. 
    If there is no Berge-$C_{r+1}$ nor Berge-$C_{r}$ but there is a Berge path of length $r$ in $C$, then we can apply Lemma \ref{lemnew3} to find a good set of size 2, a contradiction. 

    If there is no Berge path of length $r$ in $C$, then consider a longest Berge path $P$. Note that any non-defining hyperedge containing the endpoints of $P$ can only contain the defining vertices in $P$. If $P$ has fewer than $r$ defining vertices, this is impossible, thus the endpoints are contained only in defining hyperedges. 
    If $P$ has $r$ defining vertices, and the hyperedge $e$ formed by those $r$ vertices is in $\cH$ but not a defining hyperedge of this Berge path, then we can find a Berge-$C_r$, a contradiction. 
    Therefore, the endvertices of $P$ are contained only in the defining hyperedges, thus they form a good set.

We have established that each component contains a Berge path of length at least $r$, except a component that consists of only one vertex, since in that case we only find one endpoint of $P$. Clearly, two isolated vertices would form a good set. It is left to show that each component has at most $k$ vertices.
We pick a longest Berge path, then the third bullet point of Lemma \ref{lemmi} holds, thus the vertices of that path form a component. In particular, the component has at most $k$ vertices (since the number of defining vertices in the longest Berge path is at most $k$), and the same holds for the rest.

Finally, we need to show that the good sets found this way satisfy the moreover part of Lemma \ref{lemmi}. Observe that we need to deal only with the case where there is a Berge path of length more than $r+1$ in $\cH'$, while here we found good sets either when assuming that the longest Berge path has length at most $r$, or applying Lemma \ref{lemmi} itself.
\end{proof}

Therefore, we can assume that each component of $\cH'$, we say $W_1,\dots,W_t$, has order at most $k$. Moreover, we can apply Lemma \ref{lemmi}, and the third bullet point holds. This implies that for each component $W_i$ and each vertex of $W_i$, there is a Berge path of length at least $|W_i|-2$ starting at that vertex.
Observe that if a component has order less than $x$, where $x$ is the smallest integer such that $\binom{x-1}{r-1} > \binom{k-1}{r-1}/r$, then its vertices have degree at most $\binom{k-1}{r-1}/r$. In particular, $x\geq \lceil\frac{k+1}{2}\rceil+1$, since $x=\frac{k}{2}+1$ does not satisfy $r(x-1)\cdots(x-r+1)>(k-1)\cdots(k-r+1)$. 
Moreover, if there is a component of size at least 2 and at most $r+1$, then since any set $S$ of size at least two in a component of size $r+1$ has $|N_\cH(S)|\le r+1\le |S|\binom{k-1}{r-1}/r$, we find a good set of size 2, a contradiction. So we have that each component of $\cH'$, except for the at most one isolated vertex has order at least $x$ and at least $r+2$.

Let $S'$ denote the union of the component of order 1 (if exists) with $S$, then any subset of $S'$ of size $k$ that contains $S$ is a good set, thus $|S'|<k$. Let us delete those components too, i.e., we delete $S'$ from $\cH$ to obtain $\cH''$ with components $W_1,\dots,W_{t'}$ ($t'\leq t$).



\begin{clm}
    The components of $\cH''$ are in different connected components in $\cH$.
\end{clm}

\begin{proof}[\bf Proof of Claim]
Assume otherwise and consider a shortest Berge path in $\cH$ connecting $u\in W_i$ and $v\in W_j$ for some $i\neq j$. If an internal vertex $w$ of this Berge path is in $\cH''$, then there is a shorter Berge path between $u$ and $w$ and between $w$ and $v$. At least one of these Berge paths is between vertices of distinct components of $\cH''$, a contradiction. If a hyperedge of this Berge path is in $\cH''$, then one of its defining vertices $w$ is neither $u$, nor $v$ but is in $\cH''$, thus this is also impossible.

Recall that there is 
inside $W_i$ in $\cH''$
a Berge path of length at least $\lceil(k+1)/2\rceil-1$ ending in $u$. 
Similarly, there is a Berge path of length at least $\lceil(k+1)/2\rceil-1$ in $\cH''$ starting in $v$ inside $W_j$. No hyperedge or vertex is used in both of these Berge paths, since $W_i$ and $W_j$ are distinct components of $\cH''$. Moreover, no hyperedge or vertex distinct from $u$ and $v$ is used in these Berge paths or the path between $u$ and $v$, since these two Berge paths use vertices and hyperedges in $\cH''$, while the other path uses vertices and hyperedges not in $\cH''$. Combining these with the path between $u$ and $v$, we obtain a Berge path of length at least $k$, a contradiction. 
\end{proof}


Recall that if $\cH$ is disconnected, then one of its components also must be a counterexample, a contradiction to the minimality of $\cH$. Therefore, there is only one component, $\cH$ consists of $W_1$ with $r+2\le |W_1|\le k$ and $S'$ with $|S'|\le k-1$.


In particular, $n\le 2k-1$. Moreover, if $|W_1|=k$, then by the third bullet point of Lemma \ref{lemmi}, either $W_1$ forms a connected component of $\cH$ (thus $n=k$ and we are done), or $W_1$ contains a vertex of degree 1 (in which case there was no isolated vertex in $\cH'$ because these two vertices would form a good set). We move that vertex to $S'$. This way we obtain that $r+1 \le |W_1|< k$ and $|S'|\le k-1$, thus $n\le 2k-2$.





Let us order the vertices by placing the vertices of $W_1$ first, then the vertices of $S'\setminus S$ and then the vertices of $S$, in reversed order of their removal in our earlier procedure. Let $u_i$ denote the $i$th vertex in this order.
We add the vertices one by one in this order, except that we add at the same time the vertices of $S$ that were removed together. 
We claim that when we add $u_{k-1}$,  then we add at most $\binom{k-2}{r-1}/r$ hyperedges,
and when we add $u_i$ alone with $i \ge k$, then we add at most $\binom{k-1}{r-1}/r$ hyperedges. When we add $u_i$ and $u_{i+1}$ together and $i\ge k$, then we add at most $2\binom{k-1}{r-1}/r$ hyperedges, and when we add $u_i$, $u_{i+1}$ and $u_{i+2}$ together and $i\ge k$, then we add at most $3\binom{k-1}{r-1}/r$ hyperedges.



If $i\geq k$, then this holds since each vertex in $S'$ has degree at most $\binom{k-1}{r-1}/r$ in $\cH$. If we deleted $u_i$ together with $u_{i+1}$ or together with $u_{i+1}$ and $u_{i+2}$ and $i\ge k$, then this holds since we deleted a good set.


If $u_i\in S'\setminus S$ and $i<k$, then we add no hyperedges by the definition of $S'$.

Also, later, we know that we only deleted good 2-sets together except possibly the first time which may be just one or three vertices at a time. However, we reach to this only if $n=k+1$.

If $u_i\in S$ and $i<k$, then the Berge path containing $u_i$ has length at most $i-1$, thus by Lemma \ref{lemmi}, $u_i$ has degree at most $\binom{i-1}{r-1}/r$. If we add $u_{k-1}$ together with two other vertices, then we deleted a good set of size 3. This can happen only at the very first step, thus we had at most $k+1$ vertices altogether. This implies that $n=k+1$ (i.e., $q=1$), a case we have already dealt with. If we add $u_{k-1}$ together with one other vertex, thus there are at most $\binom{k-2}{r}+\binom{k-2}{r-1}/r$ hyperedges inside the set of the first $k-1$ vertices. 
Therefore, the number of hyperedges in $\cH$ is at most 
\begin{equation*}
\begin{split}\binom{k-2}{r}+ \binom{k-2}{r-1}/r+(q+1)\binom{k-1}{r-1}/r\le  \binom{k-1}{r}+(q+1)\binom{k-1}{r-1}/r=\\ \binom{k}{r}+(q+1-r)\binom{k-1}{r-1}/r.\end{split}\end{equation*}

If $q\le r-1$, then we are done. 
Assume now that $r\le q\le r+1$.

\smallskip
CASE 1. $|W_1|\le k-2$ and $r\ge 4$ or $r=q=3$. Then the number of hyperedges is at most $\binom{|W_1|}{r}+\sum_{i=|W_1|}^{k-2} \binom{i}{r-1}/r+(q+1)\binom{k-1}{r-1}/r\le \binom{k-2}{r}+\binom{k-2}{r-1}/r+(q+1)\binom{k-1}{r-1}/r\le \binom{k-2}{r}+\binom{k-2}{r-1}/r+\binom{k-1}{r-1}+(q-r+1)\binom{k-1}{r-1}/r$. Recall that $r+1\le |W_1|\le k-2$, thus $k\ge r+3$.

CASE 1.1. $r=q$.
We have $\binom{k-1}{r-1}=(k-1)\binom{k-2}{r-1}/(k-r)$, hence $\binom{k-2}{r-1}/r+(q-r+1)\binom{k-1}{r-1}/r=(1+(k-1)/(k-r))\binom{k-2}{r-1}/r=(2k-r-1)\binom{k-2}{r-1}/(r(k-r))\le \binom{k-2}{r-1}$. Therefore, the number of hyperedges in $\cH$ is at most $\binom{k}{r}$, completing the proof.

CASE 1.2. $q=r+1$ and $r\ge 4$.
We have $\binom{k-1}{r-1}=(k-1)\binom{k-2}{r-1}/(k-r)$, hence $\binom{k-2}{r-1}/r+(q-r+1)\binom{k-1}{r-1}/r=(1+(2k-2)/(k-r))\binom{k-2}{r-1}/r=(3k-r-2)\binom{k-2}{r-1}/(r(k-r))\le \binom{k-2}{r-1}+r+1$. Indeed, the inequality holds without the $r+1$ term if $r\ge 5$ or if $r=4$ and $k\ge 10$. It is straightforward to check the inequality if $r=4$ and $8\le k\le 9$. If $r=4$ and $6\le k\le 7$, we use that the number of hyperedges is at most $\binom{k-2}{r}+ \binom{k-2}{r-1}/r+(q+1)\binom{k-1}{r-1}/r$. It is easy to see that in both cases this is at most $\binom{k}{r}+\binom{q}{r}$.
Therefore, the number of hyperedges in $\cH$ is at most $\binom{k}{r}+\binom{q}{r}$, completing the proof.

\smallskip
CASE 2. $|W_1|=k-1$ and $r\ge 4$ or $r=q=3$. Note that either the vertices of $W_1$ are the defining vertices of a Berge cycle of length $k-1$ with vertices $v_1,\dots, v_{k-1}$ in this order, such that the defining hyperedges of this cycle are in $\cH'$, or $W_1$ contains a single vertex $v$ of degree 1 in $\cH'$. In the latter case, we move $v$ to $S'$ and proceed with CASE 1. In the following part of the proof, the only property of $W_1$ that we use is that it contains a Berge cycle of length $k-1$.
Since $\cH$ is connected, for every $u\in S'$ there is a shortest Berge path from $u$ to some $v\in W_1$. Since a Berge path of length $k-2$ inside $W_1$ starts from $v$, the Berge path from $u$ to $v$ must have length 1 (using that this Berge path cannot have internal defining vertices from $W_1$, thus cannot have hyperedges from $\cH'$). If a hyperedge $h$ contains two vertices $v,v'$ from $S'$, then no other hyperedge contains $v$, nor $v'$. In this case, we can delete $h,v,v'$ from $\cH$ to obtain a smaller counterexample to the statement of the theorem, a contradiction. 

We obtained that each vertex $v\in S'$ is contained in a Berge-$P_{k-1}$ in $\cH$, where the other defining vertices are the vertices of $W_1$. If $v_i$ is in a hyperedge together with $v$, then $v_{i+1}$ cannot be in a hyperedge together with $v'\in S'$ if $v'\neq v$. Let $B$ denote the set of vertices $v_i$ such that $v_i$ and $v_{i-1}$ or $v_{i+1}$ are both in some hyperedge with some $v\in S'$. Then they belong to at most one such $v$. Let $b_v$ denote the number of vertices in $B$ that form at least one hyperedge with $v$, then $\sum_{v\in S'}b_v=|B|$. For the other vertices of $W_1$, two consecutive ones cannot be in hyperedges with vertices of $S'$. Therefore, each vertex $v\in S'$ forms hyperedges with at most $\lfloor \frac{k-1-|B|}{2}\rfloor+b_v$ vertices. The number of hyperedges containing some vertex from $S'$ is at most $\sum_{v\in S'}\binom{\lfloor \frac{k-1-|B|}{2}\rfloor+b_v}{r-1}$. Clearly, this is maximized if for some $v$ we have $b_v=|B|$, thus this  is at most $q\binom{\lfloor \frac{k-1-|B|}{2}\rfloor}{r-1}+\binom{\lfloor \frac{k-1+|B|}{2}\rfloor}{r-1}$. By the convexity of the binomial function, this is maximized either at $|B|=0$ or at $|B|=k-1$.

In the second case, we have that the total number of hyperedges is at most $\binom{k-1}{r}+\binom{k-1}{r-1}=\binom{k}{r}$ and we are done. In the first case,
the number of hyperedges is at most $\binom{k-1}{r}+(q+1)\binom{\lfloor (k-1)/2\rfloor}{r-1}$. Note that compared to Lemma \ref{lemmi}, the improvement is that the defining hyperedges of the Berge path do not contain $v$, since they are in $\cH'$. When $q=r$, we have that the number of hyperedges is at most $\binom{k-1}{r}+(q+1)\binom{\lfloor (k-1)/2\rfloor}{r-1}=\binom{k-1}{r}+(r+1)\binom{\lfloor (k-1)/2\rfloor}{r-1}$. We have that $\binom{\lfloor (k-1)/2\rfloor}{r-1}\le \binom{k-1}{r-1}/2^{r-1}\le\binom{k-1}{r-1}/(r+1)$ for $r\ge 3$, thus we are done in this case. 
When $q=r+1$, we have that the number of hyperedges is at most $\binom{k-1}{r}+(q+1)\binom{\lfloor (k-1)/2\rfloor}{r-1}=\binom{k-1}{r}+(r+2)\binom{\lfloor (k-1)/2\rfloor}{r-1}$. We have that $\binom{\lfloor (k-1)/2\rfloor}{r-1}\le \binom{k-1}{r-1}/2^{r-1}\le\binom{k-1}{r-1}/(r+2)$ if $r\ge 4$, thus we are done in this case. 

\smallskip
CASE 3.
$r=3,q=4$. 

CASE 3.1. $|W_1|=k-1$. Then, we can use the following result from Case 2: every vertex in $S'$ is contained by at most $\binom{\lfloor (k-1)/2\rfloor}{r-1}$ hyperedges in $\cH$. This is the only property of $W_1$ that we use.

We claim that $\cH[W_1]$ is not hamiltonian-connected. 
Otherwise, if a hyperedge $h$ contains two vertices $v,v'$ not in $W_1$, then this leads to a contradiction as in CASE 2. 
Assume first that there are two vertices $v_1,v_2$ outside $W_1$ with hyperedges intersecting $W_1$.
Then there exist two vertices $u_1,u_2\in W_1$ and two hyperedges $h_1,h_2\in E_1$ such that $v_1,u_1\in h_1$ and $v_2,u_2\in h_2$. If $\cH[W_1]$ is hamiltonian-connected, we can find a Berge-$P_{k-2}$ from $u_1$ to $u_2$ in $\cH[A_1]$. Then the Berge-$P_{k-2}$ and $h_1,h_2$ form a Berge-$P_{k}$, a contradiction.

Assume now that there is only one vertex $v$ outside $W_1$ with a hyperedge $h$ intersecting $W_1$. Then there is a hyperedge $h'$ outside $W_1$ that contains $v$, and it is easy to see that a Berge path of length $k-2$ inside $W_1$ is extended by $h$ and $h'$ to a Berge path of length $k$, a contradiction.

We obtained that $\cH[W_1]$ is not hamiltonian-connected. 
Thus, by Theorem \ref{thm3.1}, there exists a vertex $w\in W_1$ with $d_{\cH[W_1]}(w)\leq\binom{(k-1)/2}{2}$. Note that $d_{\cH[W_1]}(v)\leq\binom{k-2}{2}$ for any $v\in W_1\backslash\{w\}$ and recall that $d_{\cH}(v)\leq\binom{\lfloor (k-1)/2\rfloor}{2}$ for any $v\in S$. 
Then the number of hyperedges of $\cH$ is 
$$\sum_{v\in S}d_{\cH}(v)+\frac{\sum_{v\in W_1}d_{\cH[W_1]}(v)}{3}\leq 5\binom{\lfloor (k-1)/2\rfloor}{2}+\frac{\binom{(k-1)/2}{2}+(k-2)\binom{k-2}{2}}{3}\leq \binom{k}{3}+\binom{4}{3},$$
thus we are done in this case. 

\smallskip
CASE 3.2. $|W_1|\le k-3$. Then we can use from CASE 1 the upper bound $\binom{|W_1|}{r}+\sum_{i=|W_1|}^{k-2} \binom{i}{r-1}/r+(q+1)\binom{k-1}{r-1}/r$ on the number of hyperedges. This is at most $\binom{k-3}{3} +\binom{k-3}{2}/3+\binom{k-2}{2}/3+  5\binom{k-1}{2}/3 \leq \binom{k}{3} + \binom{4}{3}$ and we are done.


CASE 3.3. $|W_1|=k-2$.  If $S'\neq S$, then $u_{k-1}$ was an isolated component after removing $S$, thus the number of hyperedges is at most $\binom{k-2}{3}+0+5\binom{k-1}{2}/3\leq \binom{k}{3} + \binom{4}{3}$ and we are done.
If $S'=S$, then in the deletion process deleted the two vertices $u_{k}, u_{k-1}$ together at the last step. If there are at most $k$ hyperedges incident to $u_k$ and $u_{k-1}$ when we delete them, then our upper bound on the number of hyperedges turns to $\binom{k-2}{3}+k+4\binom{k-1}{2}/3\leq \binom{k}{3} + \binom{4}{3}$ and we are done. If there are at least $k+1$ such hyperedges, then at most $k-2$ of them contains both $u_k$ and $u_{k-1}$, since the third vertex of such a hyperedge is in $W_1$. Out of the at least three other such hyperedges (each containing exactly one of $u_k$ and $u_{k-1}$), at least two contain the same vertex of $\{u_{k-1},u_{k}\}$. Without loss of generality, $h_1$ and $h_2$ contain $u_{k-1}$, such that for both $h_1$ and $h_2$, the other two vertices are in $W_1$. Let $x\neq y$ be vertices with $x\in h_1$, $y\in h_2$, and let  $W_1'=W_1\cup \{u_{k-1}\}$. If $W_1$ is hamiltonian-connected, then 
a hamiltonian path connecting $x$ and $y$ in $W_1$, together with $h_1$ and $h_2$ forms a Berge cycle of length $k-1$. Then we continue as in CASE 3.1. Recall that the only property of $W_1$ used there is a degree condition obtained in CASE 2, and in CASE 2 we used only the existence of a Berge cycle of length $k-1$.

Finally, if $W_1$ is not hamiltonian-connected, then by Theorem \ref{thm3.1} there is a vertex of degree at most $\binom{(k-2)/2}{2}$ in $W_1$. This improves our upper bound on the number of hyperedges to $\binom{(k-2)/2}{2}+ \binom{k-3}{3} + \binom{k-2}{2}/3 + 5\binom{k-1}{2}/3  \leq \binom{k}{3} + \binom{4}{3}$ and we are done. 
\end{proof}

\bigskip
\textbf{Funding}: 
The research of Cheng is supported by the National Natural Science Foundation of China (Nos. 12131013 and 12471334), the Shaanxi Fundamental Science Research Project for Mathematics and Physics (No. 22JSZ009) and the China Scholarship Council (No. 202406290241). 

The research of Gerbner is supported by the National Research, Development and Innovation Office - NKFIH under the grant KKP-133819 and by the János Bolyai scholarship.

The research of Miao is supported by the China Scholarship Council (No. 202406770056).

The research of Zhou is supported by the National Natural Science Foundation of China (Nos. 12271337 and 12371347) and the China Scholarship Council (No. 202406890088).

\end{document}